\documentclass[a4paper,12pt]{amsart}
\usepackage{amsmath,amsthm,amssymb,hyperref,graphicx,mathrsfs,mathtools}  

\newtheorem{theorem}{Theorem}[section]
\newtheorem{corollary}{Corollary}[section]
\newtheorem{lemma}{Lemma}[section]
\newtheorem{thmx}{Theorem}[section]
\newtheorem{remark}{Remark}[section]

\allowdisplaybreaks
\numberwithin{equation}{section}
\title[Zeros of certain composite polynomials]{On the zeros of certain composite polynomials and an operator preserving inequalities}
\author{N. A. Rather$^1$}
\author{Ishfaq Dar$^2$}
\address{$^{1,2}$Department of Mathematics, University of Kashmir, Srinagar-190006, India}
\email{dr.narather@gmail.com, ishfaq619@gmail.com}
\author{Suhail Gulzar$^3$}
\address{$^3$Department of Mathematics, Govt. Degree College, Anantnag-192101, India}
\email{sgmattoo@gmail.com}
\begin{document}
\maketitle
\begin{abstract}
If all the zeros of $n$th degree polynomials $f(z)$ and $g(z) = \sum_{k=0}^{n}\lambda_k\binom{n}{k}z^k$ respectively lie in the cricular regions $|z|\leq r$ and $|z| \leq s|z-\sigma|$, $s>0$, then it was proved by Marden \cite[p. 86]{mm} that all the zeros of the polynomial $h(z)= \sum_{k=0}^{n}\lambda_k f^{(k)}(z) \frac{(\sigma z)^k}{k!}$ lie in the circle $|z| \leq r ~ \max(1,s)$. In this paper, we relax the condition that $f(z)$ and $g(z)$ are of the same degree and instead
assume that $f(z)$ and $g(z)$ are polynomials of arbitrary
degree $n$ and $m$ respectively, $m\leq n,$  and obtain a generalization of this result.  As an application, we also introduce a linear operator which preserve Bernstein type polynomial inequalities.   \end{abstract}
\section{\textbf{Introduction and statement of results}}
Polynomials play an important role in many scientific disciplines and location of their zeros in particular have important applications in many areas of applied mathematics. The fundamental results concerning the relative location of the zeros of certain composite polynomials can be found
 in the comprehensive book by Marden \cite{mm} (see also \cite{rs}). Many results pertaining to the distribution of zeros of some composite polynomials can also be found in \cite[see chapter 2 of part V]{ps}. The following result concerning the comparative position of the zeros of a polynomial which is derived by the 'composition'  of two polynomials is due to Marden \cite[p. 86]{mm}.
\begin{thmx}
If all the zeros of an $n$th degree polynomial $f(z)$ lie in the circle $|z|\leq r$ and if all the zeros of the polynomial 
\begin{align*}
g(z)=\lambda_0+\binom{n}{1}\lambda_1z+\ldots+\binom{n}{n}\lambda_nz^n
\end{align*} 
lie in the circular region: 
\begin{align*}
|z|\leq s|z-\sigma|, ~~~ s>0,
\end{align*}
then all the zeros of the polynomial 
\begin{align*}
h(z)=\lambda_0f(z)+\lambda_1f^{\prime}(z)\dfrac{(\sigma z)}{1!}+\ldots+\lambda_n f^{(n)}(z)\dfrac{(\sigma z)^n}{n!}
\end{align*}
lie in the circle $|z|\leq r \max (1,s).$
\end{thmx}
Here we propose to relax the condition that the polynomials $f(z)$ and $g(z)$ are of the same degree and prove:
\begin{theorem}\label{th1}
If all the zeros of polynomial $f(z)$ of degree $n$ lie in $|z|\leq r$ and if all the zeros of the polynomial
\begin{align*}
g(z)=\lambda_0+\binom{n}{1}\lambda_1z+\ldots+\binom{n}{m}\lambda_mz^m
\end{align*}
lie in $|z| \leq s|z-\sigma|$, $s>0$, then the polynomial 
\begin{align*}
h(z)=\lambda_0f(z)+\lambda_1f^{\prime}(z)\dfrac{(\sigma z)}{1!}+\ldots+\lambda_m f^{(m)}(z)\dfrac{(\sigma z)^m}{m!}
\end{align*}
 has all its zeros in $|z| \leq r \max(1,s)$
\end{theorem}
Let $\mathcal{P}_n$ be the class of all polynomials of degree at most $n,$ then for $P\in\mathcal{P}_n$ 
\begin{align}\label{a1}
|P(z)| \leq |z^n|M \qquad \text{for} \quad |z| = 1
\end{align}
where $M = \max_{|z|=1}|P(z)|.$ According to Bernstein's inequality \cite{BNS} concerning the estimate of $|P^{\prime}(z)|$ on a unit disk, we have
\begin{align}\label{a2}
\max_{|z|= 1}\bigg|\frac{d}{dz}(P(z))\bigg|\leq \bigg|\frac{d}{dz}(z^n) \bigg|M \qquad \text{for} \quad |z|=1.
\end{align} 
This shows that inequality \eqref{a1} is preserved under differentiation for $P\in\mathcal{P}_n$. In view of this observation, it is natural to characterize the operators which preserve Bernstein-type polynomial inequalities. 
As an attempt to this characterization, we consider an operator $N$ which carries $P\in\mathcal{P}_n$ into $N[P]\in\mathcal{P}_n$ defined by
\begin{align*} 
N[P](z):=\sum\limits_{i=0}^{m}\lambda_i\left(\frac{nz}{2}\right)^i\frac{P^{(i)}(z)}{i!},
\end{align*}
where $\lambda_i,$ $ i=0,1,2,...m$ are such that all the zeros of
\begin{align*} 
\phi(z) =  \sum\limits_{i=0}^{m} \binom{n}{i}\lambda_i z^i, \quad m \leq n
\end{align*}
lie in the half plane $|z| \leq |z - \frac{n}{2}|$ and establish certain results concerning the upper-bound of $|N[P]|$ for $|z|\geq 1$. In this direction, we first present the following result:
\begin{theorem}\label{th2}
If $f(z)$ is polynomial of degree $n$ having all its zeros in $|z|\leq 1$ and $P\in\mathcal{P}_n$ such that
\begin{align*}
|P(z)|\leq |f(z)| \qquad \text{for} \quad |z|=1,
\end{align*}
then
\begin{align}\label{a3}
|N[P](z)|\leq |N[f](z)| \qquad for \quad |z|\geq 1.
\end{align}
The result is sharp and equality in \eqref{a3} holds for $P(z)=e^{i \alpha}f(z)$, $\alpha \in \mathbb{R}.$
\end{theorem}
 Taking $f(z)=M z^n$, where $ M= \max_{|z|=1}|P(z)|, $ in the above theorem, we obtain the following result.
 \begin{corollary}\label{co1}
\ If $P\in\mathcal{P}_n$ and $M= \max_{|z|=1}|P(z)|,$ then 
\begin{align} \label{x7}
|N[P](z)|\leq |N[\psi_n](z)| M\qquad for \quad |z|\geq 1,
\end{align}
where $\psi_n (z)=z^n$.
The result is sharp and equality in \eqref{x7} holds for $P(z)=e^{i \alpha}M z^n$, $\alpha \in \mathbb{R}.$
\end{corollary}
\begin{remark}\label{rm1}
Setting $\lambda_i =0, ~ i=0,1,2,...(m-1)$ in corollary \eqref{co1}, it follows that if $P\in\mathcal{P}_n$, then 
\begin{align*}
\left|\dfrac{d^{m}}{dz^{m}}(P(z))\right| \leq \left| \dfrac{d^{m}}{dz^{m}}(z^n)  \right|M, \quad for \quad |z|\geq 1,
\end{align*}
which includes inequality \eqref{a2} due to Bernstein as special case.
\end{remark}
For $P\in \mathcal{P}_n$  and not vanishing in $|z|<1,$ Paul  Erd\"{o}s conjectured that the inequality  \eqref{a2} can be replaced by 
\begin{align}\label{el}
\max_{|z|= 1}\bigg|\frac{d}{dz}(P(z))\bigg|\leq \dfrac{1}{2}\bigg|\frac{d}{dz}(z^n) \bigg|M \qquad \text{for} \quad |z|=1.
\end{align}
This result was later proved by P.D. Lax \cite{EL}.

 Next we present the following result for the class of polynomials having no zero inside the unit circle $|z|=1.$
 \begin{theorem}\label{th3}
 If $P\in\mathcal{P}_n$ and $P(z)\neq 0$ in $|z|<1$, then
\begin{align}\label{a4}
|N[P](z)|\leq \frac{1}{2}\bigg(|N[\psi_n](z)|+|\lambda_0|\bigg)M \qquad for \quad |z|\geq 1, 
\end{align}
where $M=\max_{|z|=1}|P(z)|$ and $\psi_n (z)=z^n.$ The result is best possible and equality in \eqref{a4} holds for $P(z)=az^n+b, |a|=|b|\neq 0$.
 \end{theorem}
\begin{remark}\label{rm2}
\textnormal{
Similarly as in the case of Remark \ref{rm1}, Theorem \ref{th2} implies that if $P\in\mathcal{P}_n$ and $P(z)\neq 0$ in $|z|<1$, then 
\begin{align*}
\left|\dfrac{d^{m}}{dz^{m}}(P(z))\right| \leq \dfrac{1}{2}\left| \dfrac{d^{m}}{dz^{m}}(z^n)  \right|M, \quad \text{for} \quad |z|\geq 1,
\end{align*}
which includes inequality \eqref{el} as a special case and the case $N[P](z)=\lambda_0 P(z)$ yields the following inequality by \cite{AR} for the rate of growth of a polynomial with restricted zeros
\begin{align*}
|P(z)|\leq\frac{1}{2}\bigg(|z^n|+1\bigg)\max_{|z|=1}|P(z)| \qquad |z|\geq 1.
\end{align*}
}
\end{remark}
A polynomial $P\in\mathcal{P}_n$ is said to be self-inversive polynomial if $P(z) = P^*(z)$, where $P^*(z)=z^n \overline{P\big({1}/{\overline{z}}\big)}$. It is known \cite{SIP} that the inequality \eqref{el} also holds if $P\in\mathcal{P}_n$ and is  self-inversive polynomial.
Finally we prove the following result for self-inversive polynomials.
\begin{theorem}\label{th4}
 If $P\in\mathcal{P}_n$ is a self-inversive polynomial, then
\begin{align}\label{e6}
|N[P](z)|\leq \frac{1}{2}\bigg(|N[\psi_n](z)|+|\lambda_0|\bigg) M \qquad |z|\geq 1,
\end{align}
where $M=\max_{|z|=1}|P(z)|$ and $\psi_n(z)=z^n$. Equality in \eqref{e6} holds for $P(z)=z^n+1.$
\end{theorem}
\begin{remark}\label{rm3}
\textnormal{
By taking $\lambda_i=0,$ $i=0,1,\ldots,m-1$ we get the following inequality which contains a result due to O'hara and Rodriguez \cite{SIP} as a special case.
\begin{align*}
\left|\dfrac{d^{m}}{dz^{m}}(P(z))\right| \leq \dfrac{1}{2}\left| \dfrac{d^{m}}{dz^{m}}(z^n)  \right|M, \quad \text{for} \quad |z|\geq 1.
\end{align*}}
\end{remark}
\section{\textbf{Lemmas}}
We require following lemmas for the proof of above theorems. The first lemma is due to A. Aziz \cite{WCA}.
\begin{lemma}\label{lm1}
Let $G(z_1, z_2, ....z_n)$ be a symmetric $n$-linear form of total degree $m$, $m\leq n$, in $z_1, z_2, ...., z_n$ and let C: $|z-c| \leq r$ be a circle containing the $n$ points $w_1, w_2, ...w_n$. Then in C there exists at least one point $w$ such that $$G(w, w, ....., w) = G(w_1, w_2, ...w_n).$$
\end{lemma}
The next two lemmas are required for the proofs of the Theorems \ref{th3} and \ref{th4}.
\begin{lemma}\label{lm3}
 If $P\in\mathcal{P}_n$ and $P(z)$ does not vanish in $|z|<1$, then
\begin{align*}
|N[P](z)| \leq |N[P^*](z)| \qquad for \quad |z| \geq 1,
\end{align*}
where $P^*(z)=z^n \overline{P\big({1}/{\overline{z}}\big)}$.
\end{lemma}
\begin{proof}
By hypothesis $P^*(z)=z^n \overline{P\big({1}/{\overline{z}}\big)}$, therefore $|P(z)|=|P^*(z)|$ for $|z|=1$. Also since $P(z)$ does not vanish in $|z|<1$, hence $\frac{P^*(z)}{P(z)}$ is analytic for $|z| \leq 1$ with $\big|\frac{P^*(z)}{P(z)}\big|=1$ on $|z|=1$. Therefore by the maximum modulus principle, it follows that $|P(z)| \leq |P^*(z)|$ for $|z| \geq 1$. By using Rouche's theorem,  the polynomial $P(z)- \gamma P^*(z)$ has all its zeros in $|z| \leq 1$ for every $\gamma\in\mathbb{C}$ such with $|\gamma|>1.$ Applying Theorem \ref{th1} to $P(z)- \gamma P^*(z)$ with $s=1, \sigma = \frac{n}{2}$ and noting that $N$ is linear operator, we conclude that the polynomial $N[P(z)] - \gamma N[P^*(z)]$ has all zeros in $|z| \leq 1$. This implies that $$|N[P](z)| \leq |N[P^*](z)| \qquad \text{for} \quad |z| \geq 1.$$
This completes the proof of lemma \ref{lm3}.
\end{proof}
\begin{lemma}\label{lm4}
If $P\in\mathcal{P}_n$, then for $|z| \geq 1$,
\begin{align*}
|N[P](z)| + |N[P^*](z)| \leq \big(|N[\psi_n](z)| + |\lambda_0|\big) \max_{|z|=1} |P(z)|,
\end{align*}
where $\psi_n (z) = z^n$ and $P^*(z) = z^n \overline{P\big({1}/{\overline{z}}\big)}$.
\end{lemma}
\begin{proof}
Let $M= \max_{|z|=1} |P(z)|$, then $|P(z)| \leq M$  for $|z|\leq 1.$
By Rouche's theorem it follows that the polynomial $P(z) - \gamma M$ does not vanish in $|z| < 1$ for every $\gamma \in \mathbb{C}$ with $|\gamma|>1$. Applying lemma \ref{lm3} to the polynomial $P(z) - \gamma M$, we get
\begin{align}\label{lp1}
|N[P](z)-M \gamma \lambda_0| \leq |N[P^*](z)-M \overline{\gamma}N[\psi_n](z)|, \quad \text{for} ~ |z| \geq 1,
\end{align}
where $\psi_n (z) = z^n.$ Now choosing argument of $\gamma$ such that
\begin{align*}
|N[P^*](z)-M \overline{\gamma}N[\psi_n](z)| = M|N[\psi_n](z)||\gamma|-|N[P^*](z)|,
\end{align*}
which is possible by \eqref{x7}, therefore \eqref{lp1} implies that
\begin{align*}
|N[P](z)|-M|\gamma||\lambda_0| \leq M|N[\psi_n](z)||\gamma|-|N[P^*](z)|, \quad for ~ |z| \geq 1.
\end{align*}
Letting $|\gamma| \rightarrow 1$ in above inequality, we obtain for $|z|\geq 1$,
\begin{align*}
|N[P](z)| + |N[P^*](z)| \leq \big(|N[\psi_n](z)| + |\lambda_0|\big) \max_{|z|=1} |P(z)|.
\end{align*} 
That proves lemma \ref{lm4}.
\end{proof}
\section{\textbf{Proof of the Theorems}}
\begin{proof}[Proof of Theorem \ref{th1}]
Let $w$ be an arbitrary zero of $h(z)$, then
\begin{align}\label{thp1}
\sum\limits_{k=0}^{m}\lambda_k F^{(k)}(w) \frac{\sigma^k w^k}{k!} = h(w) = 0.
\end{align}
This equation is linear and symmetric in the zeros of $F(z)$. By lemma \ref{lm1} $w$ will also satisfy the equation obtained by replacing $F(z)$ in \eqref{thp1} by $(z-\alpha)^n$ where $\alpha$ is suitably chosen point in $|z|\leq r$. That is, $w$ satisfies the equation   
\begin{align*}
\sum\limits_{k=0}^{m} \frac{\sigma^k}{k!} \lambda_k n(n-1)....(n-k+1)(w-\alpha)^{n-k} w^k = 0, 
\end{align*}
or equivalently,
\begin{align*}
(w-\alpha)^n \sum\limits_{k=0}^{m} \lambda_k \binom{n}{k} \bigg( \frac{\sigma w}{w-\alpha}\bigg)^k =0,
\end{align*}
that is,
\begin{align*}
(w-\alpha)^n g \bigg( \frac{\sigma w}{w-\alpha}\bigg) =0.
\end{align*}
Hence, we have
\begin{align*}
w-\alpha=0 \quad \text{or} \quad g \bigg( \frac{\sigma w}{w-\alpha}\bigg) =0,
\end{align*}
which implies,
\begin{align*}
w=\alpha \quad \text{or} \quad \beta =\frac{\sigma w}{w-\alpha}, ~\text{for some zero $\beta$ of $g(z).$} 
\end{align*}
 This gives,
 \begin{align*}
 w=\alpha \quad\text{or} \quad w= \frac{\alpha \beta}{\beta-\sigma}.
 \end{align*}
 Now, 
\begin{align*}
|w|=|\alpha| \leq r \quad \text{or} \quad |w|=\frac{|\alpha||\beta|}{|\beta - \sigma|} ~ \leq |\alpha|s ~ \leq rs. 
\end{align*}
Hence, it follows that
\begin{align*}
|w| \leq r ~ \max (1,s).
\end{align*}
That is, all the zeros of $h(z)$ lie in $|z| \leq r ~ \max(1,s)$. This completes the proof of Theorem \ref{th1}.

\end{proof}

\begin{proof}[Proof of Theorem \ref{th2}]
By hypothesis $f(z)$ is a polynomial of degree $n$ having all zeros in $|z|\leq 1$ and $P\in\mathcal{P}_n$ such that
\begin{align}\label{th2p1}
|P(z)|\leq |f(z)| \qquad \text{for} \quad |z|=1.
\end{align}
If $z_\nu $ is a zero of $f(z)$ of multiplicity $s_\nu$  on the unit circle $|z|=1,$ then it is evident from \eqref{th2p1} that $z_\nu$ is also a zero of  $P(z)$  of multiplicity at least $s_\nu.$ Let $g(z) = \prod\limits_{z_\nu\in\Omega}(z-z_\nu)^{s_\nu}$ where $\Omega=\{z_\nu\in\mathbb{C}:f(z_\nu)=0\wedge |z_\nu|=1 \} .$ Then again from \eqref{th2p1}, we have $$\bigg|\frac{P(z)}{g(z)}\bigg| \leq \bigg|\frac{f(z)}{g(z)}\bigg| \quad \text{for} \quad |z|=1.$$
By Rouhe's theorem for every $\gamma \in \mathbb{C}$ with $|\gamma|>1$, the polynomial $h(z)=\frac{P(z)-\gamma f(z)}{g(z)}$ has all its $n-\sum s_v$ zeros in $|z|<1$. Since the polynomial $g(z)$ has $\sum s_v$ zeros on $|z|=1$, the polynomial $h(z)g(z)=P(z)- \gamma f(z)$ has all the $n$ zeros in $|z|\leq 1$. Invoking Theorem \ref{th1} with $s=1$, $\sigma = \frac{n}{2}$ and noting that $N$ is a linear operator, it follows that all the zeros of the polynomial
\begin{align}\label{th2p2}
T(z)=N[P](z)-\gamma N[f](z)
\end{align}
lie in $|z|\leq 1$. This implies,
\begin{align}\label{th2p3}
|N[P](z)| \leq |N[f](z)| \qquad\text{ for} \quad |z|>1.
\end{align}
For if \eqref{th2p3} is not true then there exists $z_0$ with $|z_0|>1$, such that $|N[P](z)|_{z=z_0} > |N[f](z)|_{z=z_0}$, then taking $\gamma=\frac{N[P](z_0)}{N[f](z_0)}$, which is a well defined complex number with $|\gamma|>1$ and with this choice of $\gamma,$  from \eqref{th2p2}, we get  $T(z_0)=0$, $|z_0|>1.$ This clearly is a contradiction to the fact that all the zeros of $T(z)$ lie in $|z|\leq 1$. This establishes \eqref{th2p3} for $|z|>1$. For $|z|=1$, the result follows by continuity. This proves theorem \ref{th2} completely.
\end{proof}
\begin{proof}[Proof of Theorem \ref{th3}]
Since $P(z)\neq 0$ in $|z|<1$, therefore, if $P^*(z) = z^n \overline{P\big({1}/{\overline{z}}\big)}$, then by lemma \ref{lm3}, we have
$$|N[P](z)|\leq |N[P^*](z)| \quad \text{for} ~ |z|\geq 1.$$ The above inequality in conjunction lemma \ref{lm4} gives for $|z|\geq 1,$
\begin{align*}
2|N[P](z)| \leq |N[P](z)| + |N[P^*](z)| \leq \big(|N[\psi_n](z)| + |\lambda_0|\big) \max_{|z|=1} |P(z)|.
\end{align*}
This is equivalent to inequality \eqref{a4} and completes the proof of Theorem \ref{th3}.
\end{proof}

\begin{proof}[Proof of Theorem \ref{th4}]
Since $P\in\mathcal{P}_n$ is self-inversive polynomial then $P(z) = P^*(z)$ where $P^*(z)=z^n \overline{P\big({1}/{\overline{z}}\big)}$, therefore, we have
\begin{align*}
N[P](z) = N[P^*](z), \qquad\forall z\in \mathbb{ C}.
\end{align*}
Using this in lemma \ref{lm4}, we get
\begin{align*}
2|N[P](z)|\leq \bigg(|N[\psi_n](z)|+|\lambda_0|\bigg) \max_{|z|=1}|P(z)|, \qquad |z|\geq 1,
\end{align*}
where $\psi_n (z)=z^n$ and the proof of Theorem \ref{th4} is complete.
\end{proof}
\textbf{Acknowledgment:} The authors would like to thank the anonymous referee for comments and suggestions.
.
\end{document}